\documentclass[12pt]{amsart}

\usepackage{amsmath, amsfonts, amssymb, amsthm}
\usepackage{enumerate}
\usepackage{graphicx}
\usepackage[hidelinks]{hyperref}
\usepackage[utf8]{inputenc}
\usepackage{xcolor}
\definecolor{light-gray}{gray}{0.95}
\usepackage[backgroundcolor=light-gray]{todonotes}

\theoremstyle{plain}
\newtheorem{theorem}{Theorem}[section]
\newtheorem{lemma}[theorem]{Lemma}

\newtheorem{corollary}[theorem]{Corollary}

\theoremstyle{definition}

\newtheorem{notation}[theorem]{Notation}

\theoremstyle{remark}
\newtheorem{remark}[theorem]{Remark}

\numberwithin{equation}{section}

\newcommand{\R}{\mathbf{R}}
\newcommand{\N}{\mathbf{N}}
\newcommand{\Z}{\mathbf{Z}}
\renewcommand{\P}{\mathbf{P}}
\newcommand{\E}{\mathbf{E}}
\newcommand{\B}{\mathcal{B}}
\newcommand{\e}{\mathrm{e}}
\newcommand{\Oh}{\mathcal{O}}
\newcommand{\Laplace}{\Delta}
\newcommand{\dLap}{\Laplace_1}
\newcommand{\norm}[1]{\left\lVert#1\right\rVert}
\newcommand{\abs}[1]{\left|#1\right|}
\newcommand{\ceil}[1]{\left\lceil#1\right\rceil}

\newcommand{\dx}{\mathrm{d}x}
\newcommand{\floor}[1]{\left\lfloor#1\right\rfloor}

\author{Patrick W.~Dondl and Michael Scheutzow}
\title[Ballistic motion of interfaces]{Ballistic and sub-ballistic motion of interfaces in a field of random obstacles}
\date{\today}
\begin{document}  
\begin{abstract}We consider a discretized version of the quenched Edwards-Wilkinson model for the propagation of a driven interface through a random field of obstacles. Our model consists of a system of ordinary differential equations on a $d$-dimensional lattice coupled by the discrete Laplacian. At each lattice point, the system is subject to a constant driving force and a random obstacle force impeding free propagation. The obstacle force depends on the current state of the solution and thus renders the problem non-linear. For independent and identically distributed obstacle strengths with exponential moment we prove ballistic propagation (i.e., propagation with a positive velocity) of the interface if the driving force is large enough. For a specific case of dependent obstacles, we show that no stationary solution exists, but still the propagation of the front is not ballistic.
\end{abstract}
\maketitle

\section{Introduction and the main result}
\label{sec:intro}

In this article, we consider a semi-discrete model for the evolution of a driven interface subject to line tension in a random, heterogeneous, quenched environment. We first prove that if the driving force is large enough then such an interface propagates with a positive velocity---even if the random environment contains obstacles of arbitrarily large strength. Furthermore we give an example of sub-ballistic interface evolution in this class of models, when relaxing the assumptions on independence.

Let $(\Omega,\B,\P)$ be a probability space and consider the following lattice differential equation for the height $u_i \colon [0,\infty)\times\Omega \to \R$ of the $d\in\N$ dimensional interface in an ambient space of dimension $d+1$,
\begin{equation}
\label{eq:semidiscrete}
\dot{u}_i(t,\omega) = \dLap u_i(t,\omega) - f_i(u_i(t,\omega),\omega) + F,
\end{equation}
where $i\in \Z^d$, $t\ge0$, and $\omega \in \Omega$, $F\ge 0$. The initial condition is $u_i(0) = 0$. The operator $\dLap$ denotes the discrete $d$-dimensional Laplacian operator, namely $\dLap u_i = \sum_{k\in\Z^d:\norm{k-i}_1 =1} (u_k - u_i)$, where $\norm{\cdot}_1$ denotes the discrete 1-norm. The one-dimensional setting  was discussed in~\cite{Dondl:2011wo}, in this note we generalize our results to arbitrary dimension, albeit only for the (semi-)discrete evolution. 

We assume that the functions $f_i:\R \times \Omega \to [0,\infty)$, $i \in \Z^d$ are such that, disregarding infinitely fast growing unphysical solutions, equation~\eqref{eq:semidiscrete} above admits a unique solution with non-negative velocity for every $i$, that the solution depends measurably on $\omega$ for each $t \ge 0$, and that the solution furthermore follows a comparison principle. For the results in sections~\ref{sec:discrete} and~\ref{sec:cont} we also assume independence in $i$. A non-trivial (i.e., not necessarily uniformly bounded) example of such $f_i$ is given by $f_i(y,\omega) = \sum_{j=1}^\infty f_{i,j}(\omega) \phi(y-j)$ for a piecewise affine hat function satisfying $\phi(0) = 1$ and $\operatorname{supp}\phi = [-1/2,1/2]$, where $f_{i,j}$ are random variables that are independent in $i$.

The main further assumption on the $f_i$ is that they possess a finite exponential moment. As opposed to some other requirements, like complete independence, this assumption is central to our proof. Under these conditions, we can prove our main result.

\begin{theorem}\label{thm:main}
Assume in addition to the above requirements that there exists $\lambda>0$ such that 
$$
\beta:=\sup_{j \in \N_0} \E \exp\big\{\lambda \ceil{\sup_{j-.5 \le y \le j+.5}   f_0(y,\omega)}\big\} < \infty,
$$
where $\ceil{\cdot}$ denotes taking the integer ceiling of the argument.
Then there exists a non-decreasing function $V:[0,\infty) \to [0,\infty)$ which is not identically zero and 
which depends on $\lambda$ and $\beta$ only, such that for all $t>0$ we have
$$
\E \dot{u}_0(t) \ge V(F)
$$
and therefore
$$
\E \frac{u_0(t)}{t} \ge V(F).
$$
Specifically, we can choose 
$$
V(F)= \sup_{\mu>\lambda} \frac{1}{\mu}\left(\lambda (\floor{F} -2d) -\log \beta - \max\Big\{\log\frac{2}{\mu-\lambda}, \log 2e \Big\}\right),
$$
where $\floor{\cdot}$ denotes taking the integer floor of the argument.
\end{theorem}

The proof is split in two parts, first a discrete result arguing that there can be no discretized interface whose average velocity is small. The second part is an application of this result to the coupled systems of ODEs. 

\begin{remark}
Taking $\mu$ large for small values of $F$, one can see that $V(F) \ge 0$ for all $F>0$. Furthermore, as $F$ becomes large, one can take $\mu$ closer to $\lambda$ to see that there exists a constant $C$, depending only on $\lambda$, $\beta$, and $d$, such that $V(F) \ge F -\frac{1}{\lambda}\log F -C$ for all $F>1$.
\end{remark}

The main theorem also implies the following almost sure result for the point-wise velocity, excluding the existence of stationary solutions.
\begin{corollary} \label{thm:limsup}
Under the conditions of Theorem~\ref{thm:main} we also have
$$
\limsup_{n \to \infty} \frac{u_0(t_n)}{t_n} \ge V(F) \quad \mbox{almost surely}
$$
along any deterministic sequence of times $t_n \to \infty$ as $n\to \infty$. 
\end{corollary}

A model very similar to the one considered here was recently discussed in~\cite{Bodineau:2013ur}. As opposed to our model, they use a fully discrete evolution, where in each time-step the system advances by one unit at every point where the total force is positive. While some of their results are comparable to ours, they use a rigorous renormalization group approach to prove that in their model (assuming also uniformly bounded obstacles), an interface is either completely blocked (in the sense that a non-negative stationary solution exists) or that it propagates ballistically, i.e., there is no intermediate regime of sub-ballistic propagation. We show in section~\ref{sec:subbal} that this is not the case in general if the obstacles are strongly correlated.

For the present model of independent obstacles, we can only prove that there exist two critical values for the driving force: if the driving force is below the first value, the interface becomes stuck for all times. If, on the other hand, the driving force is above the second value the interface propagates with finite velocity. The first result is a simple adaptation of our methods in~\cite{Dirr:2009uw,Dirr:2010wb} and the second part is proved here. The question of whether an intermediate regime exists in this model is open.

Generally, problems of the present form (whether fully discrete, partially discrete, or fully continuous) have received considerable interest in the physics community (see for example~\cite{Kardar:1998vq,Nattermann:1992we,Narayan:1993vp,Brazovskii:2004vf}). Many connections to questions arising from physics are discussed in the aforementioned article by Bodineau and Teixeira~\cite{Bodineau:2013ur}, as well as in~\cite{Coville:2009uv}, where the first rigorous result on non-existence of stationary states was derived. 

The article is organized as follows. In section~\ref{sec:discrete}, we show non-existence of states whose velocity is too small. In the following section~\ref{sec:cont}, we apply this result to prove our theorem. Section~\ref{sec:subbal} is devoted to the example of sub-ballistic propagation. We finish with some conclusions and an outlook in section~\ref{sec:conc}.

\section{Nonexistence of slow paths}
\label{sec:discrete}
In this section, we prove the central lemma stating that in a fully discrete version of our model, one can with probability one not find any function whose average velocity is too small. Let thus now $\bar{f}_{i}(j,\omega) := \ceil{\sup_{j-.5 \le y \le j+.5} f_i(y,\omega)}$ defined for all $j\in\Z, i\in\Z^d$. For convenience, we begin by introducing some notation.
\begin{notation}
We use the following abbreviations.
\begin{itemize}
\item $Q_k := \{-k+1, \dots, k-1\}^d$, the $d$-dimensional cube of sites in $\Z^d$ of side-length $2k-1$,
\item $\B_k := \sigma(\{\bar{f}_{i} : i\in Q_k\})$, the $\sigma$-algebra generated by the random functions in $Q_k$,
\item $A\in \N$, any fixed number, later to be taken as the integer ceiling of an \textit{a priori} bound on the maximal value the functions $u_i$, solutions of~\eqref{eq:semidiscrete} can take at time $t$,
\item $P(\omega) := \{w\colon \Z^d  \to \{0, \dots, A\}, \textrm{ such that } \dLap w_i - \bar{f}_i(w_i,\omega) +F \ge 0 \textrm{ for all } i\in Z^d\}$, the set of admissible functions,
\item $P_k(\omega) := \{w\colon Q_{k+1}  \to \{0, \dots, A\}, \textrm{ such that } \dLap w_i - \bar{f}_i(w_i,\omega) +F \ge 0 \textrm{ for all } i\in Q_k\}$, the set of admissible functions within a cube $Q_{k+1}$,
\item $c_{k,d} = |Q_{k+1}\setminus Q_{k}| = (2k+1)^d-(2k-1)^d$, the size of the boundary layer around $Q_k$,
\item $N_{m,j} = \binom{j+m-1}{m-1}$, the number of ways $j\in \N_0$ can be represented as the sum of  $m$ (ordered) non-negative integers.
\end{itemize}
\end{notation}

\begin{lemma}
\label{lem:discrete}
For each $F\in\N_0$, there exists a set $\Omega_0$ of full measure such that for any  $\omega \in \Omega_0$ and any function $w\in P(\omega)$ we have 
\begin{equation}\label{ineq}
\liminf_{k \to \infty} \frac {1}{\abs{Q_k}}\sum_{i\in Q_k} \left(\dLap w_i - \bar{f}_i(w_i,\omega)+F \right) \ge \overline{V}(F),
$$
where $\overline{V}$ can be taken as 
$$
\overline{V}(F) = \sup_{\mu>\lambda} \frac{1}{\mu}\left(\lambda F -\log \beta - \max\Big\{\log\frac{2}{\mu-\lambda}, \log 2e \Big\}\right),
\end{equation}
and $\beta$ and $\lambda$ are defined in Theorem~\ref{thm:main}.
\end{lemma}
\begin{proof}
Fix $\mu > \lambda$ and consider for $k\ge 1$ the sequence of random variables
$$
Y_k := \sum_{w\in P_k} \exp\big\{ \lambda \sum_{\substack{i\in Q_k \\ r\notin Q_k \\ \norm{i-r}_1 =1}}(w_r-w_i) - \mu \sum_{i\in Q_k} \left(\dLap w_i -\bar{f}_i(w_i,\omega) +F\right) \big\}.
$$
The basic underlying idea in this definition is the following. We will show, using a martingale argument, that for sufficiently large $F$ the sequence $Y_k$ almost surely vanishes exponentially in the size of the box $Q_k$. For this decrease we can also establish a rate. Such a decrease, however, implies that as we look at larger and larger boxes around the origin, either the sum of the normal derivatives at the boundary of the box (the first term in the exponential) has to become large and negative quickly, or the sum of the velocities (the second term in the exponential) in $Q_k$ has to increase with a rate related to the one with which $Y_k$ vanishes. The first option is excluded by the non-negativity of $w$. The second option yields the average velocity (with a negative sign), after taking a logarithm and using the sum over all paths as an estimate for the supremum over all possible paths.

The first step in the proof is to relate the change in normal derivatives as $k$ increases to the addition of terms in the sum over the Laplcian. We use a discrete version of the divergence theorem, namely that 
$$
\sum_{\substack{i\in Q_k \\ r\notin Q_k \\ \norm{i-r}_1 =1}}(w_r-w_i) = 
\sum_{i\in Q_k} \dLap w_i,
$$ 
and thus
$$
Y_k = \sum_{w\in P_k} \exp\big\{ (\lambda-\mu) \sum_{i\in Q_k} \dLap w_i - \mu \sum_{i\in Q_k} \left(-\bar{f}_i(w_i,\omega) +F\right) \big\}.
$$
A calculation now yields
\begin{align*}
&\E(Y_{k+1} | \B_k) = \sum_{w\in P_k}\Bigg( \exp\big\{  (\lambda-\mu) \sum_{i\in Q_k} \dLap w_i - \mu \sum_{i\in Q_k} \left(-\bar{f}_i(w_i,\omega) +F\right)\}\cdot \\
&\quad\E\sum_{\substack{\textrm{extensions}\\ \textrm{of $w$ to $P_{k+1}$}}} \exp\big\{ \lambda \sum_{i\in Q_{k+1} \setminus Q_k} \dLap w_i  - 
\mu \sum_{i\in Q_{k+1} \setminus Q_k} (\dLap w_i - \bar{f}_i(w_i) +F) \big\} \Bigg),
\end{align*}
where the sum in the second line is taken over all admissible extensions of $w$ to functions in $P_{k+1}$. Taking now
\begin{equation} \label{eq:defgamma}
\gamma_k := \sup_{w\in P_k} \E\sum_{\substack{\textrm{extensions}\\ \textrm{of $w$ to $P_{k+1}$}}} \exp\big\{ \lambda \sum_{i\in Q_{k+1} \setminus Q_k} \dLap w_i  - 
\mu \sum_{i\in Q_{k+1} \setminus Q_k} (\dLap w_i - \bar{f}_i(w_i) +F) \big\},
\end{equation}
with the sum as above over all possible extensions, we get
\begin{equation}
\label{eq:mart}
\E(Y_{k+1} | \B_k) \le \gamma_k Y_k, \quad \textrm{for $k\ge 1$}.
\end{equation}
In order to estimate $\gamma_k$ further, we need to rearrange and count the number of possible extensions. In the sum over all admissible extensions we thus first take all extensions such that $\sum_{i\in Q_{k+1} \setminus Q_k} (\dLap w_i - \bar{f}_i(w_i) +F) = j \in \N_0$, calling these ``admissible extensions with velocity $j$'' and then sum over all $j\ge 0$. In the case that there does not exist an admissible extension with velocity $j$, we take the sum to be zero. This yields
\begin{align*}
\gamma_k &= \sup_{w\in P_k}\E\sum_{j=0}^\infty \sum_{\substack{\textrm{adm.~ext.}\\ \textrm{with vel.~$j$}}}\exp\big\{ \lambda \sum_{i\in Q_{k+1} \setminus Q_k} \dLap w_i  - \\
&\qquad\qquad\qquad\qquad\qquad\quad\;\; \mu \sum_{i\in Q_{k+1} \setminus Q_k} \left(\dLap w_i - \bar{f}_i(w_i) +F\right) \big\} \\
&= \sup_{w\in P_k}\sum_{j=0}^\infty \e^{-j(\mu-\lambda)}\e^{-\lambda c_{k,d} F} \;
\E \sum_{\substack{\textrm{adm.~ext.}\\ \textrm{with vel.~$j$}}}\exp\big\{\lambda\sum_{i\in Q_{k+1} \setminus Q_k} \bar{f}_i(w_i) \big\} \\
&\le \sup_{w\in P_k} \sum_{j=0}^\infty\e^{-j(\mu-\lambda)}\e^{-\lambda c_{k,d} F} \beta^{c_{k,d}} \sup_{\omega\in\Omega} M_{j,k,d}(\omega, w|_{Q_{k+1}}),
\end{align*}
where $M_{j,k,d}(\omega, w|_{Q_{k+1}})$ is the of the number of admissible extensions with velocity $j$, depending on the realization of the random field $f$ and on $w$ from the previous step. We also note that $w_i$ for $i \in Q_{k+1} \setminus Q_k$ is a fixed value inside the supremum, which allows us to use the assumption on the exponential moment of $f$.

The idea for estimating $M_{j,k,d}$ now is the following: given $j$, there are no more than $N_{c_{k,d},j}$ possibilities to distribute these velocities on the $c_{k,d}$ sites. With all velocities fixed, for most sites in $Q_{k+2} \setminus Q_{k+1}$ where the extension lives, the function value is determined due to the fact that $\omega$ and the velocity can be used to calculate the discrete Laplacian (if such a choice exists at all). The number of sites where we still have freedom is $\Oh(d-2)$. We thus aim for an estimate of the type $\sup M_{j,k,d} \le N_{c_{k,d},j}\cdot 1 \cdot (A+1)^{C^{d-2}}$.

\begin{figure}
\includegraphics[width=\textwidth]{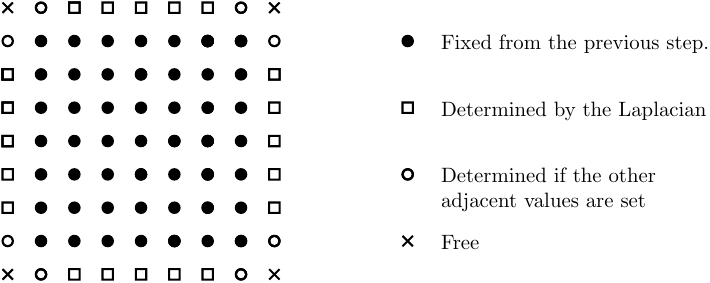}
\caption{Illustration of the extension process and the points where there is a freedom to choose the function value of the extension when the Laplacian is fixed.} \label{fig:extension}
\end{figure}
First notice that in dimension $d=1$, there is no freedom to choose any additional values for the extension if the Laplacian at the boundaries is given. Consider thus the case $d\ge 2$. Each given function value and the Laplacian on the site can be used to write one independent linear equation for the function values on the extension, therefore the remaining number of sites with freedom is $c_{k+1,d}-c_{k,d} =: \xi_{k,d}$. See the illustration in Figure~\ref{fig:extension} for the two-dimensional case. 
Note that the number of choices for each of those nodes is limited to at most $A+1$, and that $\xi_{k,d} = \Oh(k^{d-2})$.

Since the above estimate was independent of $w \in P_k$ and on $\omega\in \Omega$, this yields
$$
\gamma_k \le \sum_{j=0}^\infty\e^{-j(\mu-\lambda)}\e^{-\lambda c_{k,d} F} \beta^{c_{k,d}} N_{c_{k,d},j} (A+1)^{\xi_{k,d}}.
$$

Using the estimate $N_{m,j}\le \frac{(j+m-1)^{m-1}}{(m-1)!}\le \frac{2^{m-2}}{(m-1)!}\left(j^{m-1}+(m-1)^{m-1}\right)$, the sum can be bounded as follows. We have
$$
\sum_{j=0}^\infty N_{c_{k,d},j}\e^{-(\mu-\lambda)j} \le \frac{2^{c_{k,d}-2}}{(c_{k,d}-1)!}
\left(\sum_{j=0}^\infty j^{c_{k,d}-1}\e^{-(\mu-\lambda)j} + (c_{k,d}-1)^{c_{k,d}-1}\frac{1}{1-\e^{-(\mu-\lambda)}}\right),
$$
which, using
\begin{align*}
\sum_{j=0}^\infty j^{c_{k,d}-1}\e^{-(\mu-\lambda)j} &\le \sum_{j=0}^\infty \int_j^{j+1} x^{c_{k,d}-1} \e^{-(\mu-\lambda)(x-1)}\,\dx \\
&=\frac{\e^{\mu-\lambda}}{(\mu-\lambda)^{c_{k,d}}}\Gamma(c_{k,d}) \\
&= \frac{\e^{\mu-\lambda}}{(\mu-\lambda)^{c_{k,d}}}(c_{k,d}-1)!
\end{align*}
yields 
$$
\sum_{j=0}^\infty N_{c_{k,d},j}\e^{-(\mu-\lambda)j} \le \max\left\{\frac{2^{c_{k,d}}\e^{\mu-\lambda}}{(\mu-\lambda)^{c_{k,d}}}, \frac{(2c_{k,d}-2)^{c_{k,d}-1}}{(c_{k,d}-1)!}\frac{1}{1-\e^{-(\mu-\lambda)}} \right\} 
$$
We thus have 
\begin{align*}
\log \gamma_k &\le \xi_{k,d}\log(A+1) + c_{k,d}(\log \beta-\lambda F)+  \\
&\quad + 
\max\Big\{\mu-\lambda+c_{k,d}\log\frac{2}{\mu-\lambda},\\
& \quad\quad\quad  (c_{k,d}-1) \log 2\e + \log C - \log\left(1-\e^{-(\mu-\lambda)}\right) \Big\}
\end{align*}
for some constant $C$.

Equation~\eqref{eq:mart} together with the boundedness of $\gamma_k$ and the almost sure finiteness of $Y_1$ establishes that $ \frac{Y_k}{\prod_{j=1}^{k-1}\gamma_j}$ for $k\ge 2$ is a non-negative supermartingale. From Doob's martingale convergence theorem we therefore find that 
$$
\lim_{k\to\infty} \frac{Y_k}{\prod_{j=1}^{k-1}\gamma_j} = C(\omega) < \infty
$$
on a set $\Omega_0 \subset \Omega$ of full measure. Note furthermore that
$$
\frac{\log Y_k}{\mu \abs{Q_k}} \ge \sup_{w\in P_k} \Big\{\frac{\lambda}{\mu\abs{Q_k}}  \sum_{\substack{i\in Q_k \\ r\notin Q_k \\ \norm{i-r}_1 =1}}(w_r-w_i) - \frac{1}{\abs{Q_k}}\sum_{i\in Q_k} \left(\dLap w_i -\bar{f}_i(w_i,\omega) +F\right) \Big\}
$$
and thus
\begin{align*}
& \inf_{w \in P(\omega)}\liminf_{k\to\infty} \frac {1}{\abs{Q_k}}\sum_{i\in Q_k} \left(\dLap w_i - \bar{f}_i(w_i,\omega)+F \right) \\
&\quad \ge \liminf_{k\to \infty} \frac{-1}{\mu \abs{Q_k}} \sum_{i=1}^k \log \gamma_j \\
&\quad \ge \frac{1}{\mu} \left(\lambda F -  \log \beta- \max\Big\{\log\frac{2}{\mu-\lambda}, \log 2\e \Big\}\right)
\end{align*}
where we have used that $\abs{Q_k} = \sum_{i=1}^k c_{i,d}$ and dropped all terms that are of lower order than $\abs{Q_k}$. In particular, these are the terms in $\gamma_i$ that are of lower order than $c_{i,d}$ as well as $C(\omega)$ and the first sum inside the exponent in $Y_k$, which vanishes in the limit due to the boundedness of $w$. This proves the lemma.
\end{proof}

\section{Application to the continuous evolution problem} \label{sec:cont}
The lemma from the above section allows us to complete the proof of the main theorem.
\begin{proof}[Proof of Theorem~\ref{thm:main}]
Assume that the statement in the theorem is untrue. Then there exist $F \ge 0$ and some $t_0$  
such that $\E\dot u_0(t_0) < V(F)$. By our independence assumptions on the field $f$, the processes 
$u_i(t_0)$, $\dot u_i(t_0)$, $i \in \Z^d$ are stationary and ergodic and take values in $[0,\infty)$. 
We write $u_i$ instead of $u_i(t_0)$. By Birkhoff's ergodic theorem, we have 
$$
\E\dot u_0 = \lim_{k \to \infty} \frac {1}{\abs{Q_k}}\sum_{i\in Q_k} \left(\dLap u_i - f_i(u_i,\omega)+F \right)
$$
almost surely. However, taking $w_i$ to be $u_i$ rounded to the closest integer, we find
\begin{align*}
&\lim_{k \to \infty} \frac {1}{\abs{Q_k}}\sum_{i\in Q_k} \left(\dLap u_i - f_i(u_i,\omega)+F \right) \ge \\
&\quad\liminf_{k \to \infty} \frac {1}{\abs{Q_k}}\sum_{i\in Q_k} \left(\dLap w_i -2d - \bar{f}_i(w_i,\omega)+\floor{F} \right) \ge \overline{V}(\floor{F}-2d) = V(F)
\end{align*}
by Lemma~\ref{lem:discrete}.
\end{proof}
The almost sure statement about the velocities can be derived by the following argument.
\begin{proof}[Proof of Corollary~\ref{thm:limsup}]
Consider, for a fixed sequence of times $t_n \to \infty$, the random variables
$$
A_i(\omega) := \limsup_{n\to\infty} \frac{u_i(t_n)}{t_n}, \quad i \in \Z^d,
$$
noting that $A_i$ is stationary, ergodic and bounded from above and below by $F$ and $0$, respectively. Furthermore, we have $\E(A_i) \ge V(F)$, by Fatou's lemma. By the non-negativity of the velocity and $f_i$, it follows that $\dLap u_i(t,\omega) \ge -F$ for all $t\ge 0$ and almost all $\omega$, and therefore $\dLap A_i(\omega) \ge 0$ for almost all $\omega$.

Now let $\xi_i:=\E\dLap A_i$. By stationarity, $\xi_i$ is constant in $i$ and we write $\xi := \xi_0$. By the discrete divergence theorem, boundedness of $A_i$ and ergodicity of $\dLap A_i$ imply that $\xi = 0$ and since $\dLap A_i(\omega) \ge 0$ for almost all $\omega$ we have $\dLap A_i(\omega) = 0$ almost surely and for all $i \in \Z^d$. This yields that $A_i(\omega)$ is a bounded, ergodic, and stationary process whose realizations are almost surely harmonic. Thus, $A_i(\omega)$ is almost surely constant in $i$ and therefore $A_0(\omega)$ is almost surely equal to its expected value. The desired result follows.
\end{proof}

\section{An example for non-ballistic evolution} \label{sec:subbal}
In the following we construct a counterexample showing that non-existence of a stationary solution does not necessarily imply a positive velocity. For simplicity, we first consider a \emph{fully-discrete} evolution problem, where the interface height $u$ at discrete times $k\in\N_0$ is given by
$$
u_i(k+1,\omega)-u_i(k,\omega) =  S(\dLap u_i(k,\omega) - Q_{i,u_i(k,\omega)}(\omega) + F),
$$
for a given random obstacle field $Q_{i,j}(\omega) \in \{0,1\}$ for $i\in\Z$, $j\in\N_0$ and with initial condition $u_i(0) = 0$. The evolution law $S\colon \R\to \{0,1\}$ is given by $S(a) = 1$ for (strictly) positive $a$, zero otherwise. In the following, we fix $F=1/2$, noting that the interface will not move at a point $(i,j)$ where there is an obstacle (i.e, $Q_{i,j} = 1$) and the interface is flat (i.e., $u_{i-1}(k) = u_{i}(k) = u_{i+1}(k) = j$). If, on the other hand, either the site $(i,j)$ does not have an obstacle or the interface possesses an upward kink in the sense that $\dLap u_i(k,\omega) \ge 1$, the interface will advance in that time-step. We remark that this process follows a comparison principle, that is, considering $\tilde{Q}_{i,j} \ge Q_{i,j}$ for all $i\in\Z$, $j\in\N_0$ and $\tilde{u}_i(k) \le u_i(k)$ for some $k\in\N_0$ and all $i\in\Z$, where $\tilde{u}$, $u$ evolve according to the given process with $\tilde{Q}_{i,j}$ and $Q_{i,j}$, respectively, we have $\tilde{u}_i(l) \le u_i(l)$ for all $l\ge k$ and all $i\in\Z$.

Let now $(N_i)$, $i \in \Z$ be a discrete stationary renewal point process, i.e., $N_i$ is a $\Z$-valued random variable for each $i \in \Z$ such that the random variables $Y_i:=N_i-N_{i-1}$ are iid and strictly positive and the point process is stationary. Let $F$ be the distribution function of $Y_1$. Stationarity of the point process implies that $\E Y_1<\infty$ but imposes no further constraints upon $F$ (other than that $F$ is the distribution function of an $\N$-valued random variable). In the following, let $Y$ be a random variable with distribution function $F$.

For a given $F$ we take independent copies of the process $N$ in each row $j=0,1,2,...$ and say that at each lattice point $(i,j)$ which is {\em not} an element of the point process is an obstacle of size 1 and the other lattice points are free of obstacles. For $i \in \Z$ and $j\in \N_0$, we define thus $Q_{i,j}=1$ if there is an obstacle in row $j$ at location $i$ and $Q_{i,j}=0$ otherwise. For $j\in \N_0$, let $X_j:=\min\{i \in \N_0:Q_{i,j}=0\} \wedge \min\{i \in \N_0:Q_{-i,j}=0\}$. Clearly, the $X_j$ are iid. By choosing $F$ appropriately, we can ensure, that the $X_j$ have an arbitrarily long tail.

Clearly, in each row of the obstacle field there almost surely exist infinitely many holes (on either side of the origin), i.e., for any $j\in\N_0$ there are infinitely many $i\in \N$ such that $Q_{i,j}=0$. We immediately see that no stationary non-negative solution can exist: such a stationary solution would have to be completely flat (otherwise it would necessarily have to have an upward kink), but there is no row without a hole in the obstacle field. Furthermore, for any $M\in\N$, $i\in \Z$, we can calculate a random upper bound for time $k$ such that $u_i(k)\ge M$: start at the lattice point $(i,M-1)$ going right or left until the first hole in the obstacle field appears. From there start going in the row below, again left or right, until the next hole is found. The total number of steps ($+M$) that have to be taken until a hole at row zero is found is the sought after bound. Thus, for any $i\in \Z$, we have $\liminf_{k\to\infty} u_i(k) = +\infty$.

In order to ensure that $\liminf_{n \to \infty} u_0(n)/n =0$ almost surely for the associated discrete time interface model, it suffices to assume that  $\P\{X_1 \ge i^2\}\ge 1/i$ for all $i \in \N$ (say). By the second Borel-Cantelli Lemma, this implies that almost surely, infinitely many of the $X_j$ are greater or equal to $j^2$. Noticing that $X_j \ge j^2$ implies that $u_0(j^2) \le j$, we obtain $\liminf_{n \to \infty} u_0(n)/n =0$ almost surely.

Next, we want to investigate the behaviour of  $\limsup_{n \to \infty} u_0(n)/n$. We first show that if the $X_j$ are sufficiently heavy-tailed, then we also have $\limsup_{n \to \infty} u_0(n)/n=0$. For a given sequence $\alpha_1<\alpha_2<...$ of positive integers, we find some $F$ as above such that 
$$
\sum_{m=1}^{\infty} \big(\P\{X_1 \le \alpha_m\}\big)^m < \infty.
$$
The first Borel-Cantelli Lemma now yields that almost surely we have $\max\{X_1,...,X_m\}>\alpha_m$ for all but finitely many $m$. Choosing for example $\alpha_m=m^2$, we see that $u_0(m^2)\le m$ for all but finitely many $m$, so the claim follows.

For the semi-discrete model~\eqref{eq:semidiscrete}, a similar construction is possible. We take $d=1$, consider $Q_{i,j}$ as above and set $f_i(j+1) = Q_{i,j}$. Then we extend each $f_i$ to a function on the whole real line in a piecewise affine manner, requiring that each piecewise affine hat function is supported in a small $[i-\delta,i+\delta]$-interval around the integers,  fixing $\delta$ later. Setting $F=0.9$, a supersolution $v$ satisfying $\limsup_{t\to\infty} v_0(t)/t = 0$ can be found: start with $v_i<1$ constant in $i\in\Z$  such that $F-f_i(v_i) = -0.05$ for any $i$ where $Q_{i,0} =1$. Now, for any $i\in \Z$ where the fully discrete model described above would jump, simply replace the jump by a motion with velocity $\dot{v_i} = F+2$ for a very short time and then a jump such that the total distance travelled is $1$. This evolution can be continued by always propagating the sites that would jump in the discrete model. We note that due to the fact that for any $i\in\Z$ where the interface was stuck, it was in fact strictly stuck in the sense that the total force acting on $v_i$ is $-0.05$ -- thus a short enough motion of a neighboring site will not induce a positive force at $i$. Furthermore, for our chosen velocity, this evolution is always a supersolution, since the total right hand side of the equation never exceeds $2+F$. We note that in this evolution the time for $v_i$ to reach a fixed height is the same, modulo a constant factor, as for the fully discrete model.

The construction of a subsolution is slightly more involved. We start with $v_i=1-\delta$ constant in $i$, i.e., such that $f_i(v_i) = 0$ for all $i$. Now one can slowly propagate $v_i$ for any $i$ where $Q_{i,0} =0$, until the point where $v_i=1-\delta+0.4$. The additional force acting on neighboring sites through the discrete Laplacian is now large enough that they can propagate as well and pass the obstacles as long as $\delta$ is sufficiently small, e.g., $0<\delta<0.05$. The process of propagation of $v_i$, now jointly with its two nearest neighbors can continue until $v_i$ has reached the value $2-\delta$ and $v_{i-1}=v_{i+1} = 1-\delta+0.6$. Now, again, the force acting on the next nearest neighbors is strong enough so they can start propagating, and thus the evolution can be continued in this local fashion. In order to remove ambiguity in the evolution, we assume here that in each row there is only exactly one obstacle missing, and we are always in the situation that missing obstacles are not nearest or next-nearest neighbors on the lattice. Note that this still provides us with a subsolution, and the time $k$ until $v_i(k) \ge M$ can be calculated in the same way as for the fully discrete model.

\section{Conclusions} \label{sec:conc}
In this note, we have extended our depinning result from~\cite{Dondl:2011wo} to the case of arbitrary dimension in a semi-discrete model of coupled ordinary differential equations. A careful inspection of the proof shows that one can furthermore extend our results to obstacle strengths coupled over a finite distance: if there exists $L>0$ such that sets of obstacles are independent if their distance (in the first $d$-dimensions) is above $L$, one can still obtain similar estimates for the velocity.

The case of the fully continuous model on $\R^d$, however, remains open. Further unresolved issues are whether we have $\liminf_{t\to\infty} \frac{u_0(t)}{t} > 0$ almost surely for sufficiently large $F$, the relaxation of the result to obstacles with fat tails, as well as whether a regime of sub-ballistic propagation (i.e., vanishing velocity, but propagation of the interface to $+\infty$ everywhere) can exist in our models with independent obstacles. As mentioned above, for a specific fully discrete evolution model this last question was answered recently~\cite{Bodineau:2013ur}.


\end{document}